\documentclass[10pt]{article}
\usepackage{amsmath,amsthm}
\newcommand{\rul}[1]{}
\newcommand{\bhat}[1]{brown}
\newcommand{\ghat}[1]{gray}
\newcommand{\Alice}[1]{Alice}
\newcommand{\Bob}[1]{Bob}
\newcommand{\Charlie}[1]{Charlie}
\newtheorem{theorem}{Theorem}[section]
\newtheorem{lemma}[theorem]{Lemma}

\usepackage{pict2e,fancybox,rotating,graphicx}

\begin{document}

\title{Yet Another Hat Game}
\author{Maura B.\ Paterson\\ Department\ of Economics, Mathematics and Statistics\\
Birkbeck, University of London\\
Malet Street, London WC1E 7HX, UK\\
\and 
Douglas~R.~Stinson\thanks{research supported by NSERC discovery grant 203114-06}\\
David R. Cheriton School of Computer Science\\University of
Waterloo\\
Waterloo Ontario, N2L 3G1, Canada}
\date{\today}

\maketitle

\begin{abstract}
Several different ``hat games'' have recently received a fair amount of
attention.
Typically, in a hat game, one or more players are required to correctly guess their
hat colour when given some information about other players' hat colours.
Some versions of these games have been motivated
by research in complexity theory and have ties to well-known
research problems in coding theory, 
and some variations have led to interesting new research.

In this paper, we review {\bf Ebert's Hat Game} \cite{Ebert,EV}
which garnered a considerable amount of publicity in the late 90's and early
00's \cite{Robinson}, 
and the
{\bf Hats-on-a-line Game} \cite{Aravamuthan,TOAD}. Then 
we introduce a new
hat game which is a ``hybrid'' of these two games
and provide an optimal strategy for playing the new game.
The optimal strategy is quite simple, but the proof involves
an interesting combinatorial argument.
\end{abstract}

\section{Introduction}

In this introduction, we review two popular hat games and mention some related work. 
In Section \ref{new.sec}, we introduce our new game and give a complete solution
for it. In Section \ref{comments.sec}, we make some brief comments.

\subsection{Ebert's Hat Game}

The following hat game was posed in a 1998 computer science PhD thesis
 by Todd Ebert \cite{Ebert}.
There are three players: 
 \Alice{15}, 
 \Bob{15}, and 
 \Charlie{15}.
 The three players enter a room and a gray  or brown  hat is placed on each person's head.
The colour of each hat is determined by a coin toss, with the outcome of one coin toss having no effect on the others.

Each person can see the {other players' hats} but not his or her own hat.
No communication of any sort is allowed, except for an {initial strategy session} before the game begins. 
Once they have had a chance to look at the other hats, the players must {simultaneously} guess the colour of their own hats, or pass. So each player's response is one of 
``\ghat{5}'',  ``\bhat{5}''  or ``{pass}''.
The group shares a hypothetical \$1,000,000 prize if {at least one player guesses correctly} and {no players guess incorrectly}.

It is not hard to devise a strategy that will win 50\% of the time.
For example, \Alice{7.5}  could guess ``\ghat{5}''  while \Bob{7.5} and \Charlie{7.5} 
pass.  Is it possible to do better? Clearly, any guess has only a 50\% chance of being correct. If more than one player guesses, then the probabilities are reduced:
the probability that {two guesses are correct} is 25\%, and the probability that
{three guesses are correct} is 12.5\%. Hence, it seems at first glance
that it is impossible to win more than 50\% of the time. 

However, suppose each player uses the following rule: If he  observes two hats of 
the {same colour} (i.e., \ghat{5} -- \ghat{5} or \bhat{5} -- \bhat{5}), then he {guesses the opposite colour}. Otherwise, when two hats of {different colours} are observed, he passes.
To analyse the probability of winning when using this strategy, 
we consider all possible distributions of hats. 
There are $2 \times 2 \times 2 = 8$ cases to consider.
In each case, we can figure out if the players win or lose. The probability of winning is equal to {the number of winning configurations
divided by eight}.
In the following Table \ref{tab.1}, 
we provide an analysis of all eight cases. Boldface type is used
to indicate correct votes.

\begin{table}[tb]
\caption{Analysis of Ebert's hat game for three players}
\label{tab.1}
\begin{center}
\begin{tabular}{ccc|ccc|c}
\multicolumn{3}{c|}{configuration} & \multicolumn{3}{c|}{guesses} & \multicolumn{1}{c}{outcome}\\ \hline
\rul{7}\bhat{5} &\bhat{5} & \bhat{5} & {\ghat{5}} & {\ghat{5}} & {\ghat{5}} & {{lose}}\\ \hline
\rul{7}\bhat{5} &\bhat{5} & \ghat{5} & & & {\bf \ghat{5}} &  {{win}}\\ \hline
\rul{7}\bhat{5} &\ghat{5} & \bhat{5} & & {\bf \ghat{5}} & & {{win}}\\ \hline
\rul{7}\bhat{5} &\ghat{5} & \ghat{5} & {\bf \bhat{5}} & & & {{win}}\\ \hline
\rul{7}\ghat{5} &\bhat{5} & \bhat{5} & {\bf \ghat{5}} & & & {{win}}\\ \hline
\rul{7}\ghat{5} &\bhat{5} & \ghat{5} & & {\bf \bhat{5}} & &  {{win}}\\ \hline
\rul{7}\ghat{5} &\ghat{5} & \bhat{5} & & & {\bf \bhat{5}} &  {{win}}\\ \hline
\rul{7}\ghat{5} &\ghat{5} & \ghat{5} & {\bhat{5}} & {\bhat{5}} & {\bhat{5}} & {{lose}}
\end{tabular}
\end{center}
\end{table}

The group wins in {six out of eight cases}, so their probability of
winning is $6/8 = 3/4 = 75\%$.
Observe that {each individual guess} is correct with a 50\% probability.
Among the eight cases, there are {six correct guesses} and 
{six incorrect guesses}.
The six correct guesses occurred in {six different cases}, while the 
six incorrect guesses were squeezed into {two cases}. 
This is why the probability of winning is much higher than 50\%, even though
each guess has only a 50\% chance of being correct!


Here is another way to describe the optimal 3-player strategy:
\begin{itemize}
\item specify  \bhat{5}-\bhat{5}-\bhat{5} and \ghat{5}-\ghat{5}-\ghat{5}
as {\it bad configurations}.
\item If a player's hat colour could result in a bad configuration, then that player
guesses {the opposite colour}. 
\item If a player's hat colour could not result in a bad configuration, then that player {passes}. 
\end{itemize}

{Strategies for more players are based on this idea of specifying certain
appropriately chosen {bad
configurations} and then using a similar strategy as in the 3-player game.}
The bad configurations are obtained using  {\it Hamming codes}, 
which are perfect single error correcting codes.
For every integer $m \geq 2$, there is a Hamming code of length $n = 2^m-1$ containing 
$2^{2^m-m-1} = 2^{n-m}$ 
{codewords}.

In a Hamming code, every non-codeword can be changed into {exactly one}
codeword by changing one entry. (This property allows the 
Hamming code to correct {any single error} that occurs during transmission.)
If the configuration of hats is not a codeword, then there is a {unique position} $i$
such that changing entry $i$ creates a codeword. {Player $i$ will therefore guess
correctly and every other player will pass.}
If the configuration of hats is a codeword, then {everyone will guess incorrectly}.
Thus the group wins if and only if the configuration of hats is {not} a codeword.

Since there are $2^{n-m}$ codewords and $2^n$ configurations in total,
the success probability is $1 - 2^{-m} =
1 - 1/(n+1)$. It can be proven fairly easily that
this success probability is optimal, and can be attained only when a perfect
1-error correcting code exists. More generally, any strategy for this hat game on an arbitrary number $n$ of players is
``equivalent'' to a {\it covering code} of length $n$, 
and thus optimal strategies (for any number 
of players) are known if and only if optimal covering codes are known (see \cite{Lenstra}
for additional information).

\subsection{Hats-on-a-line}

Another popular hat game has $n$ players standing in a line. Hats of two colours
(gray and brown) are distributed randomly to each player. Each player
$P_i$ $(1 \leq i \leq n$) can only see the hats worn by 
players $P_{i+1}, \dots , P_n$ (i.e., the players ``ahead of'' $P_i$ in the line).
Each player is required to guess their hat colour, and they 
guess in the order $P_1, \dots , P_n$. The objective is to maximise the number of
correct guesses \cite{TOAD,Aravamuthan}. 

Clearly the first player's guess will be correct with probability 50\%, no matter
what her strategy is. However, a simple strategy can be devised in which 
players $P_2 , \dots , P_n$ always guess correctly by making use of information
gleaned from prior guesses.

As before, suppose that $0$ corresponds to \ghat{5} and $1$ corresponds to 
\bhat{5}. Let $c_i$ denote the colour of player $P_i$'s hat, $1 \leq i \leq n$.
Here is the strategy:
\begin{itemize}
\item
$P_1$ knows the values $c_2, \dots , c_n$ (she can see the hats belonging to
$P_2, \dots , P_n$).
$P_1$ provides as her guess the value \[g_1 = \sum_{i=2}^{n} c_i \bmod 2.\]
\item $P_2$ hears the value $g_1$ provided by $P_1$ and
$P_2$ knows the values $c_3, \dots , c_n$. Therefore $P_2$ can compute
\[ c_2 = g_1 - \sum _{i=3}^{n} c_i \bmod 2.\] 
$P_2$'s guess is $c_2$, which is correct.
\item For any player $P_j$ with $j \geq 2$, $P_j$
hears the values $g_1, c_2, \dots , c_{j-1}$ provided by $P_1 ,\dots , P_{j-1}$ respectively, and
$P_j$ knows the values $c_{j+1}, \dots , c_n$. Therefore $P_j$ can compute
\[c_j = g_1 - \sum _{i \in \{ 2, \dots , n\} \backslash \{j\}} c_i \bmod 2.\] 
$P_j$'s guess is $c_j$, which is correct.
\end{itemize}

It is not hard to see that the same strategy can be applied 
for an arbitrary number of colours, $q$, where $q > 1$.
The colours are named $0, \dots , q-1$ and all computations
are performed modulo $q$. If this is done, then $P_1$ has probability $1/q$ of 
guessing correctly, and the remaining $n-1$ players will always guess correctly.
Clearly this is optimal.
   
\subsection{Related Work}

A few years prior to the introduction of
Ebert's Hat Game, in 1994, a similar game was described by
Aspnes, Beigel, Furst  and Rudich \cite{ABFR}. In their version of
the game, players are not allowed to pass, and the objective is for a 
majority of the players to guess correctly. For the three-player game, it is easy
to describe a strategy that will succeed with probability $3/4$, just as in
Ebert's game:
\begin{itemize}
\item Alice votes the opposite of Bob's hat colour;
\item Bob votes the opposite of Charlie's hat colour; and
\item Charlie votes the opposite of Alice's hat colour.
\end{itemize}

This game is analysed in Table \ref{tab.2}, where the outcomes
for all
the possible configurations are listed.

\begin{table}[tb]
\caption{Analysis of the majority hat game for three players}
\label{tab.2}
\begin{center}
\begin{tabular}{ccc|ccc|c}
\multicolumn{3}{c|}{configuration} & \multicolumn{3}{c|}{guesses} & \multicolumn{1}{c}{outcome}\\ \hline
\rul{7}\bhat{5} &\bhat{5} & \bhat{5} & {\ghat{5}} & {\ghat{5}} & {\ghat{5}} & {{lose}}\\ \hline
\rul{7}\bhat{5} &\bhat{5} & \ghat{5} & {\ghat{5}} & {\bf \bhat{5}} & {\bf \ghat{5}} &  {{win}}\\ \hline
\rul{7}\bhat{5} &\ghat{5} & \bhat{5} & {\bf \bhat{5}} & {\bf \ghat{5}} & {\ghat{5}} & {{win}}\\ \hline
\rul{7}\bhat{5} &\ghat{5} & \ghat{5} & {\bf \bhat{5}} & {\bhat{5}} & {\bf \ghat{5}} & {{win}}\\ \hline
\rul{7}\ghat{5} &\bhat{5} & \bhat{5} & {\bf \ghat{5}} & {\ghat{5}} & {\bf \bhat{5}} & {{win}}\\ \hline
\rul{7}\ghat{5} &\bhat{5} & \ghat{5} & {\bf \ghat{5}} & {\bf \bhat{5}} & {\bhat{5}} &  {{win}}\\ \hline
\rul{7}\ghat{5} &\ghat{5} & \bhat{5} & {\bhat{5}} & {\bf \ghat{5}} & {\bf \bhat{5}} &  {{win}}\\ \hline
\rul{7}\ghat{5} &\ghat{5} & \ghat{5} & {\bhat{5}} & {\bhat{5}} & {\bhat{5}} & {{lose}}
\end{tabular}
\end{center}
\end{table}

It is also possible to devise a strategy for the majority hats game that uses Hamming
codes. 
We basically follow the presentation from \cite{Buhler}.
The idea, which is due to Berlekamp, is to associate a strategy for $n$ players with 
an orientation of the edges of the $n$-dimensional cube $\{0,1\}^n$. Each player's
view corresponds in a natural way to an edge of the cube, and that player's guess
will be determined by the head of the edge, as specified by the orientation. 

If $n$ is a power of 2 minus 1, then there is Hamming code of length $n$.
Direct all the edges of the cube incident with a codeword away from
the codeword.  The remaining edges form an eulerian graph on the vertices that
are not codewords; these edges can be directed according to any eulerian circuit.

The number of correct guesses for a given configuration is equal to the indegree of the 
corresponding vertex. From this observation, it is not difficult to see that 
any codeword is a losing configuration for this strategy --- in fact, 
every guess will be incorrect.  If the configuration of hats is not a codeword, then
there will be $(n+1)/2$ correct guesses and $(n-1)/2$ incorrect guesses. So the success
probability is $1 - 1/(n+1)$, as in the Ebert hat game, and this can again be shown to be optimal.

Many other variations of the hat game have been proposed. We complete this section 
by briefly mentioning some of them.

\begin{itemize}
\item Hats could be distributed according to a non-uniform
probability distribution (\cite{Guo}).
\item 
Usually, it is  stipulated that each player gets a single guess as to his or her hat colour; however, allowing players to have multiple guesses 
has also been considered (\cite{ABFR}). 
\item When sequential responses are used, 
it may be the case that players can hear all the previous
responses (we call this {\it complete auditory information}), or only some of them,
as in \cite{Aravamuthan}.
\item Some games seek 
to guarantee that a certain minimum number of correct guesses are made, regardless of the configuration of hats, e.g., in an adversarial setting (\cite{ABFR,Winkler}).
\end{itemize}

In general, players' strategies can be deterministic or nondeterministic (randomized).
In the situation where hat distribution is done randomly, it suffices to consider only
deterministic strategies.  However, in an adversarial setting, an optimal strategy may
require randomization.


\section{A New Hats-on-a-line Game}
\label{new.sec}

When the second author gave a talk to high school students about Ebert's Hat Game, 
one student asked
about sequential voting. It is attractive to consider sequential voting especially in 
the context of the Hats-on-a-line Game, but in that game the objective is different than in 
Ebert's game. A natural ``hybrid'' game  would allow sequential voting, but retain the same
objective as in Ebert's game. So we consider the following
new hats-on-a-line game specified as follows:
\begin{itemize}
\item hats of $q>1$ colours are distributed randomly;
\item visual information is restricted to the hats-on-a-line scenario;
\item sequential voting occurs in the order $P_1, \dots , P_n$ with abstentions allowed; and
\item the objective is 
that at least one player guesses correctly and no player guesses incorrectly.
\end{itemize}
We'll call this game the {\bf New Hats-on-a-line Game}.

First, we observe that it is sufficient to consider strategies where only one
player makes a guess. If the first player to guess is incorrect, 
then any subsequent guesses 
are irrelevant because the players have already lost the game. On the other hand, 
if the first  player to guess is correct, then the players will win if all the
later players pass.

We consider the simple strategy presented in Table \ref{t-gray}, 
which we term the {\bf Gray Strategy}.
The Gray Strategy can be applied for any number of colours
(assuming that gray is one of the colours, of course!).

\begin{table}[tb]
\caption{The Gray Strategy}
\label{t-gray}
\begin{center}
\begin{tabular}{|p{4in}|}
\hline
Assume that gray is one of the hat colours. For each player $P_i$ ($1 \leq i \leq n$), 
when it is player $P_i$'s turn, if he can see at least one gray hat, he passes;
otherwise, he guesses ``gray''. \\ \hline
\end{tabular}
\end{center}
\end{table}

It is easy to analyse the success probability of 
the Gray Strategy: 

\begin{theorem}
\label{gray.thm}
The success probability of the Gray Strategy for the
New Hats-on-a-line Game with $q$ hat colours and $n$ players is 
$1 -  \left((q-1)/{q}\right)^{n}$.
\end{theorem}

\begin{proof}
The probability that $P_1$ sees no gray hat
is $((q-1)/q)^{n-1}$. In this case, her guess of ``gray'' is correct with probability
$1/q$. If $P_1$ passes, then there is at least one gray hat among the remaining $n-1$ 
players. Let $j = \max \{ i: P_i \text{ has a gray hat}\}$. Then players
$P_1, \dots , P_{j-1}$ will pass and player $P_j$ will correctly guess ``gray''.
So the group wins if player $P_1$ passes. Overall, the probability of winning
is 
\[ \frac{1}{q} \times \left(\frac{q-1}{q}\right)^{n-1} + 1 \times \left( 1 -  \left(\frac{q-1}{q}\right)^{n-1} \right) =  1 -  \left(\frac{q-1}{q}\right)^{n}  .\]
\end{proof}

The main purpose of this section is to show that the Gray Strategy is an optimal strategy.
(By the term ``optimal'', we mean that the strategy has the maximum possible
probability of success, where the maximum is computed over  all possible 
strategies allowed by the game.)
We'll do two simple special cases before proceeding to the general proof.
(The proof of the general case is independent of these two proofs, but the
proofs of the special cases are still of interest due to their simplicity.)

We first show that the Gray Strategy is optimal if $q=2$. In this proof and
all other proofs in this section, we can restrict our attention without loss of generality 
to deterministic strategies.

\begin{theorem}
The maximum success probability for any strategy for the
New Hats-on-a-line Game with two hat colours and $n$ players is $1 -  2^{-n}$.
\end{theorem}

\begin{proof}
The proof is by induction on $n$. For $n=1$, the result is trivial, as
any guess by $P_1$ is correct with probability $1/2$. So we can assume $n > 1$.

Suppose there are $c$ configurations of $n-1$ hats for which player $P_1$
guesses a colour. We consider two cases:
\begin{description}
\item [case 1: $\mathbf{c \geq 1}$] \mbox{\quad}  \smallskip \\
There are $c$ cases where $P_1$'s guess is correct with
probability $1/2$. Therefore the probability of an incorrect guess by $P_1$ 
is
\[ \frac{1}{2} \times \frac{c}{2^{n-1}} \geq \frac{1}{2^n}.\] 
\item [case 2: $\mathbf{c = 0}$] \mbox{\quad} \smallskip \\
Since player $P_1$ always passes, the game reduces to an $(n-1)$-player game, 
in which the probability of winning is at most  $1 -  2^{-n+1}$, by induction.
\end{description}
Considering both cases, we see that the probability of winning is
at most $\max \{ 1 -  2^{-n}, 1 -  2^{-n+1}\} = 1 -  2^{-n}$.
\end{proof}

We observe that the above proof holds even when every player has complete
visual information, as the restricted visual information in the hats-on-a-line
model is not used in the proof. 

We next prove optimality for
the two-player game for an arbitrary number of hat colours, as follows.
 
\begin{theorem}
The maximum success probability for any strategy for the
New Hats-on-a-line Game with $q$ hat colours and two players is 
\[1 -  \left( \frac{q-1}{q} \right)^{2} = \frac{2q-1}{q^2}.\]
\end{theorem}

\begin{proof}
Suppose that player $P_1$ guesses her hat colour for $r$ out of the 
$q$ possible colours for 
$P_2$'s hat that
she might see. Any guess she makes is correct with probability $1/q$. 

We distinguish two cases:
\begin{description}
\item [case 1: $\mathbf{r=q}$] \mbox{\quad}  \smallskip \\
If $r = q$,
then the overall success probability is $1/q$. 
\item [case 2: $\mathbf{r<q}$] \mbox{\quad}  \smallskip \\
In this case,  player
$P_1$ passes with probability $(q-r)/q$. Given that $P_1$ passes, $P_2$ knows
that his hat is one of $q-r$ equally possible colours, so his guess will be correct with
probability $1 / (q-r)$.  Therefore the overall success probability is
\[\frac{1}{q} \times \frac{r}{q}  + \frac{1}{q-r} \times \frac{q-r}{q}  = 
\frac{r}{q^2}  + \frac{1}{q}.\]
To maximise this quantity, we take $r = q-1$. This yields a success probability of
$(2q-1)/q^2$. 
\end{description}
Case 2 yields the optimal strategy 
because $(2q-1)/q^2 > 1/q$ when $q > 1$. 
\end{proof}

\subsection{The Main Theorem}

Based on the partial results proven above, it is tempting to 
conjecture that the maximum success strategy is 
$1 -  \left((q-1)/{q}\right)^{n}$,
for any integers $n>1$ and $q>1$.
In fact, we will prove that this is always the case.

The proof is done in two steps. 
A strategy is defined to be \emph{restricted} if
the any guess made by any player other than the first player
is always correct.  First, we show that any optimal 
strategy must be a restricted strategy.
Then we prove optimality of the Gray Strategy by considering only 
restricted strategies.

In all of our proofs, we denote the colour of $P_i$'s hat by $c_i$, $1 \leq i \leq n$.
The $n$-tuple $(c_1, \dots , c_n)$ is the {\it configuration} of hats.

\begin{lemma} 
\label{restricted.lem}
Any optimal strategy for the New Hats-on-a-line Game is a restricted strategy.
\end{lemma}

\begin{proof}
Suppose $\mathcal{S}$ is an optimal strategy for the New Hats-on-a-line Game
that is not restricted.
If player $P_1$ passes, then the outcome of the game is determined by the 
$(n-1)$-tuple $(c_2, \dots , c_n)$, which is known to $P_1$. Since $P_1$ knows
the strategies of all the players, she can determine exactly which $(n-1)$-tuples
will lead to incorrect guesses by a later player. Denote this 
set of $(n-1)$-tuples by $F$. Because $\mathcal{S}$ is not restricted,
it follows that $F \neq \emptyset$. 

We create a new strategy 
$\mathcal{S}'$ by modifying $\mathcal{S}$ as follows:
\begin{enumerate}
\item If $(c_2, \dots , c_n) \in F$, then $P_1$ guesses an arbitrary colour (e.g., $P_1$ 
could guess ``gray'').
\item If $(c_2, \dots , c_n) \not\in F$, then proceed as in $\mathcal{S}$.
\end{enumerate}
It is easy to see that $\mathcal{S}'$ is a restricted strategy.
The strategies $\mathcal{S}$ and $\mathcal{S}'$ differ only in what 
happens for configurations $(c_1, \dots , c_n)$ where $(c_2, \dots , c_n) \in F$.
When $(c_2, \dots , c_n) \in F$, $\mathcal{S}'$ will guess correctly with probability
$1/q$. On the other hand, $\mathcal{S}$ always results in an incorrect guess when 
$(c_2, \dots , c_n) \in F$.
Because $|F| > 1$, the success probability of $\mathcal{S}'$
is greater than the
success probability of $\mathcal{S}$. This contradicts the optimality of
$\mathcal{S}$ and the desired result follows. 
\end{proof}

Now we proceed to the second part of the proof.

\begin{lemma}
\label{restr-prob.lem}
The maximum success probability for any restricted strategy for the
New Hats-on-a-line Game with $q$ hat colours and $n$ players is 
$1 -  \left((q-1)/{q}\right)^{n}$.
\end{lemma}

\begin{proof}
Suppose an optimal restricted strategy $\mathcal{S}$ is being used.
Let $A$ denote the set of $(n-1)$-tuples $(c_2, \dots , c_n)$ for which $P_1$ guesses;
let $B$ denote the set of $(n-1)$-tuples for which $P_1$ passes and $P_2$ guesses (correctly);
and let $C$ denote the set of $(n-1)$-tuples for which $P_1$ and $P_2$ both  pass.
Clearly every $(n-1)$-tuple is in exactly one of $A$, $B$, or $C$, so
\begin{equation}
\label{eq1} |A| + |B| + |C| = q^{n-1}.
\end{equation}

Now construct $A'$ ($B'$, $C'$, resp.) from $A$ ($B$, $C$, resp.) by deleting the first
co-ordinate (i.e., the value $c_2$) from each $(n-1)$-tuple. $A'$, $B'$ and $C'$ are treated as multisets.
We make some simple observations:
\begin{description}
\item [(i)] $B' \cap C' = \emptyset$. This beacuse $P_2$'s strategy is determined by
the $(n-2)$-tuple $(c_3, \dots , c_n)$.
\item [(ii)] For each $(c_3, \dots , c_n) \in B'$, there are precisely $q-1$ occurrences
of $(c_3, \dots , c_n) \in A'$.  This follows because player $P_2$ can be 
guaranteed to guess correctly
only when his hat colour is determined uniquely. 
\item [(iii)] $A' \cap C' = \emptyset$. This follows from the optimality of the strategy $\mathcal{S}$.
(The existence of an $(n-1)$-tuple $(c_2, \dots , c_n) \in A$ such that $(c_3, \dots , c_n) \in C'$ contradicts the optimality of $\mathcal{S}$:
$P_1$ should pass, for this configuration will eventually lead to a correct guess by a later
player.)
\end{description}
We now define a restricted 
strategy $\mathcal{S'}$ for the 
$(n-1)$-player game with players $P_2, \dots , P_n$ (here $P_2$ is the ``first'' player).
The strategy is obtained by modifying $\mathcal{S}$, as follows:
\begin{enumerate}
\item $P_2$ guesses (arbitrarily) if $(c_3, \dots , c_n) \in A' \cup B'$ and 
$P_2$ passes if $(c_3, \dots , c_n) \in C'$. (This is well-defined in view of the three
preceding observations.)
\item $P_3, \dots , P_n$ proceed exactly as in strategy $\mathcal{S}$.
\end{enumerate}
Since the set of $(n-2)$-tuples for which $P_2$ passes is the same in both of
strategies $\mathcal{S}$ and $\mathcal{S'}$, it follows that $P_3, \dots , P_n$
only make correct guesses in  $\mathcal{S'}$, and therefore  $\mathcal{S'}$ is restricted.

Let $\beta_n$ denote the maximum number of $(n-1)$-tuples for which the first player
passes in an optimal restricted strategy. We will prove that
\begin{equation}
\label{eq5}
\beta_n \leq q^{n-1} - (q-1)^{n-1}.
\end{equation}
This is true for $n=2$, since $\beta_2 \leq 1$.

Now we proceed by induction on $n$. 
We will use a few equations and inequalities.
First, from (ii), it is clear that
\begin{equation}
\label{eq2} |A| \geq (q-1)|B|.
\end{equation}
Next, because $\mathcal{S'}$ is a restricted strategy for $n-1$ players, we have 
\begin{equation}
\label{eq3} |C| \leq q \beta_{n-1}.
\end{equation}
Finally, from the optimality of $\mathcal{S}$, it must be the case that
\begin{equation}
\label{eq4} |B| + |C| = \beta_n.
\end{equation}
Applying (\ref{eq1}), (\ref{eq2}), (\ref{eq3}) and (\ref{eq4}), we have
\begin{eqnarray*}
\beta_n & = & |B| + |C| \\
& = & q^{n-1} - |A| \\
& \leq & q^{n-1} - (q-1)|B| \\
&=& q^{n-1} - (q-1)(\beta_n - |C|)\\
& \leq & q^{n-1} - (q-1)\beta_n + q(q-1)\beta_{n-1},
\end{eqnarray*}
from which we obtain
\[ \beta_n  \leq  q^{n-2} + (q-1) \beta_{n-1}. \]
Applying the induction assumption, we see that
\[ \beta_n  \leq  q^{n-2} + (q-1) (q^{n-2} - (q-1)^{n-2}) = q^{n-1} - (q-1)^{n-1},\] 
showing that (\ref{eq5}) is true.

Finally, using (\ref{eq5}), the success probability of $\mathcal{S}$ is computed to be
\begin{eqnarray*}
 \lefteqn{\mathrm{Pr}[P_1 \text{ passes}] + \frac{1}{q} \times \mathrm{Pr}[P_1 \text{ guesses}]}\\
&=&  \mathrm{Pr}[P_1 \text{ passes}] + \frac{1}{q} \times ( 1- \mathrm{Pr}[P_1 \text{ passes}])\\
&=&  \frac{1}{q} + \mathrm{Pr}[P_1 \text{ passes}]  \times \left( 1- \frac{1}{q} \right)\\
&\leq & \frac{1}{q}  + \frac{\beta_n}{q^{n-1}} \times \left( 1- \frac{1}{q} \right) \\
&\leq & \frac{1}{q} +  \left( \frac{q^{n-1} - (q-1)^{n-1}}{q^{n-1}} \right) \times \left( 1- \frac{1}{q} \right) \\
&=& 1 - \left(\frac{q-1}{q} \right)^n.
 \end{eqnarray*}
\end{proof}

Summarizing, we have proven our main theorem.

\begin{theorem}
The Gray Strategy for the
New Hats-on-a-line Game with $q$ hat colours and $n$ players is 
optimal.
\end{theorem}

\begin{proof}
This is an immediate consequence of 
Theorem \ref{gray.thm} and Lemmas \ref{restricted.lem} and 
\ref{restr-prob.lem}.
\end{proof}

\section{Comments}
\label{comments.sec}

It is interesting to compare Ebert's Hat Game, the Hats-on-a-line Game and 
the New Hats-on-a-line Game. The optimal solutions to Ebert's game are easily 
shown to be  equivalent to 
covering codes.  There are many open problems 
concerning these combinatorial structures, so the
optimal solution to Ebert's game is not known in general. 
The optimal solution to the Hats-on-a-line Game is a simple
arithmetic strategy, and it is obvious that the strategy is optimal.
We have introduced the New Hats-on-a-line Game as a hybrid of the two preceding
games. The optimal strategy is very simple, but the proof of optimality
is rather delicate combinatorial proof by induction. This game does not 
seem to have any connection to combinatorial structures such as covering codes.
The analysis of these three games utilize different techniques. At the present 
time, there does not
appear to be any kind of unified approach that is appropriate for 
understanding these games 
and/or other types of hat games.

\end{document}